\documentclass[12pt, reqno]{amsart}

\usepackage[utf8]{inputenc}
\usepackage[T1]{fontenc}
\usepackage{tikz,tikz-3dplot}
\usepackage{amsmath,amsfonts,amssymb,amsmath,latexsym,mathrsfs}
\usepackage{shuffle}
\usepackage[all]{xy}
\usepackage{stackengine}
 
\usepackage{enumerate}

\usepackage{hyperref}
 
\usepackage{tabu}
\usepackage{tikz-cd} 

\usepackage{comment}
\usepackage{url}
\usepackage{tikz}

\usepackage{xfrac}

\newtheorem{theorem}{Theorem}[section]

\newtheorem{proposition}[theorem]{Proposition}

\newtheorem{example}[theorem]{Example}

\begin{document}

\title{Quandles with isomorphic representation categories}

\author{Mohamad \textsc{Maassarani} }
\maketitle
\begin{abstract} We give two types of examples where two non isomorphic finite quandles have isomorphic enveloping groups. This implies that two non-isomorphic finite quandles can share isomorphic representation categories (as additive and tensor categories simultaneously).
\end{abstract}
\section*{Introduction and main results}
A quandle is a set $Q$ equipped with a binary operation $\triangleright: Q\times Q \to Q$ satisfying some axioms. To a group one can attach functorially a conjugacy quandle. A quandle representation of a quandle $Q$ over a vector space $V$ is a quandle morphism from $Q$ to the conjugacy quandle of $GL(V)$, i.e. $\rho(x\triangleright y)=\rho(x)\rho(y)\rho(x)^{-1}$. One can attach to each quandle $Q$ an enveloping group $G(Q)$ and a universal map $Q\to G(Q)$. The categories of finite dimensional representaitons of $Q$ and $G(Q)$ (group representations) are isomorphic as additive and tensor categories simultaniously. The definitions of objects and properties we mentioned are given in section $1$ of these notes.\\\\
To an non injective quandle $Q$ ($Q\to G(Q)$ not injective) one can attach a smaller injective quandle $Q_{conj}$. $Q$ and $Q_{conj}$ share isomorphic enveloping groups (\cite{GQ}) and hence have isomorphic representation categories. One can wonder wether we can have non isomorphic finite injective quandles with isomorphic enveloping groups and hence isomorphic representation categories. We give a positive answer using the known notion of a dual of a quandle. This is done in section $2$.\\\\
\textbf{AI tools :} The part concerning dual quandles was obtained using Google's AI mode and the free version of claude (Claude Sonnet 5.5). The information was checked by the author.
\section{Definitions an properties}
\subsection{Quandle}
A quandle is a set $Q$ equipped with a binary operation $\triangleright$  such that :
\begin{itemize}
\item[-] $x\triangleright x=x$ for $x\in Q$.
\item[-] For all $x,y \in Q$ there exist a unique $z\in Q$, such that $x \triangleright z=y$.
\item[-] $x \triangleright( y\triangleright z)= (x \triangleright y) \triangleright (x \triangleright z)$, for $x,y,z \in Q$. 
\end{itemize}
The conjugacy quandle $Conj(G)$ of a group $G$ is the set $G$ equipped with the binary operation defined by $x\triangleright y= xyx^{-1}$  for $x,y\in G$.\\\\  
A map between two quandles $(Q_1,\triangleright_1)$ and $(Q_2 ,\triangleright_2)$ satisfying : $$f(z \triangleright_1 w)=f(z) \triangleright_2 f(w),$$
for $z,w \in Q$ is called a quandle mophism. It is an isomorphism if $f$ is a bijection. Composition of quandle morphisms gives quandle morphisms. A group morhism is a quandle morphism if the groups are considered as conjugacy quandles.
 
\subsection{Quandle representation}
Folowing \cite{rep} a representation of a quandle $Q$ is a vector space $V$ and a quandle morphism $\rho: Q \to Conj(\mathrm{GL}(V))$, i.e. : 
$$\rho(x\triangleright y)=\rho(x) \rho(y) \rho(x)^{-1},$$
for all $x,y \in Q$. We define intertwining operators as for group representations.
\subsection{Enveloping group}
The enveloping group $G(Q)$ of a quandle $Q$ is the group :
$$G(Q)=\langle e_x \in Q \vert e_xe_ye_x^{-1}=e_{x\triangleright y} , \text{ for } x,y\in Q\rangle.$$
The map $\varphi_Q : Q \to G(Q)$ assigning to $x\in Q$ the corresponding generator is a quandle morphism for the conjugcay quandle structure on $G(Q)$. For any quandle morphism $f:Q\to Conj(G)$ there is a unique group morphism $\tilde{f}:G(Q)\to G$ such that $f=\tilde{f}\circ\varphi_Q$.
 \subsection{The category of quandle representations}
For $Q$ a quandle we will denote by $Rep(Q)$ the category whose object are finite dimensional quandle representations of $Q$ and whose morphisms are intertwinning operators. We will denote by $Rep(G(Q))$ the categorie of finite dimensional representations of $G(Q)$ with morphisms intertwinning operators. The category $Rep(G(Q))$ is naturally an additive category and a tensor category. Indeed, one can form as for groups the direct sum of two quandle representations and the tensor product (using the diagonal $Q\to Q\times Q$) of two quandle representations. The category $Rep(Q)$ is an additive and a tensor category.\\\\
Let $\rho : Q \to GL(V)$ be a quandle representation, by the universal property of the map $\varphi_Q$ from $Q$ to the envelopping group $G(Q)$ there is a unique group representation $\rho_{G(Q)}: G(Q)\to GL(V)$ such that $\rho_{G(Q)}=\rho \circ \varphi_Q$.
\begin{proposition}
\item[1)] The assignement $\rho \mapsto \rho_{G(Q)}$ is a $1$ to $1$ correspondance between quandle representations of $Q$ and group representations of $G(Q)$.
\item[2)] If $h$ is an intertwinning operator between two quandle representations of $Q$ it is an intertwinning operater between the corresponding representations of $G(Q)$.
\item[3)] The assignement $Hom_{Rep(Q)}(\rho,\rho') \to Hom_{Rep(G(Q))}(\rho_{G(Q)},\rho'_{G(Q)})$ in $2)$ is bijective. 
\end{proposition}

\begin{proposition}\label{Rep}
\item[1)] The assignement of the previous proposition $\rho \mapsto \rho_{G(Q)}$ and $h\mapsto h$ is a category isomorphism $F$ between $Rep(Q)$ and $Rep(G_Q)$. 
\item[2)] $F$ is an additive functor and a tensor functor.
\end{proposition}

\section{Quandles with isomorphic enveloping groups}
 
A quandle $Q$ is called injective if $\varphi(Q):Q\to G(Q)$ is injective. For a non injective quandle $Q$ the image of $Q$ by $\varphi_Q$ is an injective subquandle $Q_{Conj}$ of $Conj(G(Q))$ (\cite{GQ}).  One can find an exemple of a non injective quandle on $3$ elements in loc. cit..

\begin{proposition}
Let $Q$ be a non injective finite quandle. 
\item[1)] The quandles $Q$ and $Q_{Conj}$ are not isomorphic.
\item[2)] We have a category isomorphism between $Rep(Q)$ and $Rep(Q_{conj})$. This isomorphism is an isomorphism of addive categories and tensor categories.
\end{proposition}
\begin{proof}
The quandle $Q_{Conj}$ is not isomorphic to $Q$ since $\vert Q_{Conj} \vert < \vert Q \vert$, but $G(Q_{Conj})$ and $G(Q)$ are isomophic (\cite{GQ}). Hence, we have an isomorphism of (additive and tensor) categories between $Rep(G(Q))$ and $Rep(G(Q_{Conj}))$. $2)$ follows from proposition \ref{Rep}.
\end{proof}
 This gives a first type example of non isomorphic finite quandles with isomorphic representation categories and raises the question whether this could happen with two injective finite quandles. The answer is yes as we will show.\\\\
Let $Q$ be a quandle  with binary operation $\triangleright$. The dual quandle $Q^*$ of $Q$ is the set $Q$ equipped with the binary operation $\triangleright^{-1}$ defined by :
$$ x\triangleright^{-1}y=z \Longleftrightarrow x\triangleright z=y,$$
for $x,y \in Q$. 
\begin{proposition}
Let $Q$ be a quandle.
\item[1)] $G(Q)$ and $G(Q^*)$ are isomorphic.
\item[2)] $Rep(Q)$ and $Rep(Q^*)$ are isomorphic categories. The isomorphism is an isomorphism of additive and tensor categories.
\item[3)] $Q$ is injective if and only if $Q^*$ is injective.
\end{proposition}
\begin{proof}
We prove $1$. Recall that the enveloping groups of $Q$ and $Q*$ are the group :
$$G(Q)=\langle e_x \mid e_{x\triangleright y}=e_xe_ye_x^{-1}\rangle\quad \text{and}\quad  G(Q^*)=\langle f_x \mid f_{x\triangleright^{-1}y}=f_xf_yf_x^{-1}\rangle.$$
The fact that $z=x\triangleright^{-1}y$ means that $y=x\triangleright z$. Hence, we have the relation :
$$f_z=f_xf_{x\triangleright z} f_x^{-1}, $$
in $G(Q^*)$. Writing this relation differently, we get :
$$f_{x\triangleright z}^{-1}=f_x^{-1}f_{z}^{-1}f_x.$$
In $G(Q^*)$. The correspondance $e_x\mapsto f_x^{-1}$ maps the relation $e_{x\triangleright z}=e_xe_ze_x^{-1}$ onto $ f_{x\triangleright z}^{-1}=f_x^{-1}f_{z}^{-1}f_x$. Hence, we have a morphism $\theta : G(Q) \to G(Q^*)$ induced by the correpondance $e_x\mapsto f_x^{-1}$. Since $(Q^*)^*=Q$ we have a morphism $\theta^*:G(Q^*)\to G(Q)$ given by the correspondance $f_x\mapsto e_x^{-1}$. The maps $\theta$ and $\theta^*$ are inverse to each other and hence they are group isomorphism. This proves $1)$. $2)$ follows from $1)$ and proposition\ref{Rep}. To prove $3)$ notice that that the universal maps $\varphi_Q : Q \to G(Q)$ and $\varphi_{Q^*}:Q^*\to G(Q^*)$ satisfy the equation $$\theta \circ \varphi_Q(x)=\varphi_{Q^*}(x)^{-1}.$$ Hence $\varphi_Q$ is injective if and only if $\varphi_{Q^*}$ is.
\end{proof}
It follows that to provide two non isomorphic finite injective quandles with isomorphic enveloping groups and hence isomorphic representation categories one can exhibit a finite injective quandle not isomorphic to its dual quandle. 

\begin{example}
Let $A$ be an abelian group and $t$ be a natural number prime to $\vert A\vert$. The Alexander quandle $Al(A,t)$ is the set $A$ equipped with the binary operation
$$x\triangleright y= (1-t)x+t y$$
for $x,y$ in $A$. $Alex(A,t)$ is an injective quandle $($\cite{ak}$)$. The dual of $Al(\mathbb{Z}/5\mathbb{Z},2)$ is the quandle $Al(\mathbb{Z}/5\mathbb{Z},3)$. The quandles $Al(\mathbb{Z}/5\mathbb{Z},2)$ and $Al(\mathbb{Z}/5\mathbb{Z},3)$ are not isomorphic $($\cite{ho}$)$. More generally one can consider $Al(\mathbb{Z}/pZ,t)$ and its dual with $p$ prime and $t^2$ not equal to $1$ modulo $p$.
\end{example}

 \end{document}